\documentclass{amsart}

\newtheorem{theorem}{Theorem}[section]

\newtheorem{lemma}[theorem]{Lemma}

\newtheorem{remark}[theorem]{Remark}
\newtheorem{remarks}[theorem]{Remarks}

\title[CMC hypersurfaces of Riemannian and Lorentzian groups]{On CMC hypersurfaces of Riemannian and Lorentzian groups}

\author{A. Caminha}
\address{Departamento de Matem\'atica, Universidade Federal do Cear\'a, Fortaleza,
Cear\'a, Brazil. 60455-760}
\email{caminha@mat.ufc.br}

\subjclass[2010]{Primary: 53C42; secondary: 53C30.}

\keywords{Lie groups; bi-invariant metrics; Riemannian groups; Bernstein-type problems; constant mean curvature}

\thanks{The author is partially supported by grants $302463/2010-0$ and $473618/2011-7$ of CNPq.}

\begin{document}

\begin{abstract}
In this paper, we study the geometry of a connected oriented cmc Riemannian hypersurface $M$ of a Riemannian or Lorentzian group $G$
Let $\mathfrak g$ be the Lie algebra of $G$. If $G$ is Riemannian and $M$ is compact and transversal to an element of $\mathfrak g$, 
we show that it is a lateral class of an embedded Lie subgroup of $G$; we also do this if $G$ is Lorentzian, provided $M$ has 
sufficiently large mean curvature. If $G$ is Riemannian semisimple and $M$ is compact, we prove that $M$ has degenerate Gauss map and 
minimal nullity at least $1$. We also extend the above results to the case where $M$ is complete and noncompact, in the following way: 
for a Riemannian $G$, we show that an $M$ which is minimal is either transversal to an element of $\mathfrak g$, hence stable, or has degenerate 
Gauss map and minimal nullity at least $1$; for an $M$ which is of cmc and is transversal to an element of $\mathfrak g$, if we ask the 
immersion to be proper and have bounded second fundamental form, then $M$ is also a lateral class of an embedded Lie subgroup of $G$, 
provided a certain growth condition on the size of the corresponding Gauss map is satisfied. Finally, for a Lorentzian group $G$, 
with sectional curvatures bounded from above on Lorentzian planes, we extend a result of Y. Xin, proving that a complete $M$, which is 
of cmc, is either 
totally geodesic or a hyperbolic space form, provided it is transversal to a timelike element of $\mathfrak g$, has large enough mean 
curvature and bounded hyperbolic Gauss map. We also present some examples of Riemannian and Lorentzian groups, as well as codimension one 
embedded subgroups, which are relevant to our results.
\end{abstract}

\maketitle

\section{Introduction}\label{section:Introduction}

%Although several authors (e.g. \cite{Daniel:07} and \cite{Mira:12} and the references therein) have sistematically studied cmc surfaces in 
%three dimensional Lie groups with left invariant metrics in recent years, much less work has been done on the geometry of cmc hypersurfaces 
%of higher dimensional Lie groups (nevertheless, see \cite{Lira:12}).

Classical geometric Lie group theory assures that every compact and every semisimple Lie group can be turned into a 
semi-Riemannian (Riemannian, in the compact case) group, i.e., can be furnished with a bi-invariant metric tensor. However, if one stops 
for a moment and tries to search for examples of submanifolds, even hypersurfaces, of such a group, ones possessing interesting geometric 
properties related to curvature (constant mean curvature, for example), then one gets easily stuck on trivial examples, like those of 
lateral classes of Lie subgroups of such a group, furnished with the induced metric.

In this paper we give an explanation of why this is so, presenting sufficient conditions for the validity of the converse of the above 
situation in the cases of a Riemannian or Lorentzian group and a complete Riemannian cmc hypersurface of it. 

Firstly, we consider a Riemannian group $G$ and a complete connected oriented cmc hypersurface $M$ of $G$, which is transversal to an 
element of the Lie algebra of $G$. In the compact case (cf. Theorem \ref{thm:main theorem}), this suffices to show that $M$ is a lateral 
class of an embedded Lie subgroup of $G$; in the complete noncompact case (cf. Theorem \ref{thm:second theorem}), if we ask 
$\varphi$ to be proper and to have bounded second fundamental form, then we reach the same conclusion, provided a certain growth condition
on the size of the corresponding Gauss map is satisfied. We also show (cf. Theorem \ref{thm:semisimple}) that a compact connected oriented 
cmc hypersurface of a semisimple Riemannian group $G$ has degenerate Gauss map and minimal nullity at least $1$. Finally,
we extend this last result to the complete noncompact case, proving (cf. Theorem \ref{thm:stability}) that an $M$ which is minimal 
is either transversal to an element of $\mathfrak g$, hence stable, or has degenerate Gauss map and minimal nullity at least $1$. 

Secondly, we treat the case of a Lorentzian group $G$ and a complete connected oriented cmc spacelike hypersurface $M$ of $G$, 
also under the assumption of transversality with respect to an element of the Lie algebra of $G$. If $M$ is compact and has large enough mean curvature, we prove 
(cf. Theorem \ref{thm:the compact Lorentzian case}) that it is a lateral class of an embedded Lie subgroup of $G$. If $G$
has sectional curvatures of Lorentzian planes bounded from above, we prove (cf. Theorem \ref{thm:extending a theorem of Xin}) 
that a complete such $M$ is either totally geodesic or a hyperbolic space form, provided it has large enough mean curvature and bounded 
hyperbolic Gauss map.

The results of Theorems \ref{thm:main theorem} and \ref{thm:second theorem} extend, to the context of Riemannian and Lorentzian groups,
previous works of several authors (cf. \cite{Albujer:11}, \cite{Caminha:11}, \cite{Caminha:10} and the references therein)
on Berstein-type problems for complete hypersurfaces of Riemannian and Lorentzian warped products. 
Theorem \ref{thm:extending a theorem of Xin} extends, to Lorentzian groups, 
a result of Y. Xin (cf. \cite{Xin:91}) on cmc complete hypersurfaces of the Lorentz-Minkowski space.

The paper is organized in the following manner: in Section \ref{section:Riemannian groups}, we establish some notations and 
collect a few facts on Lie groups and Lie algebras that will be needed later. In Section \ref{section:Support functions on hypersurfaces},
we give a detailed account, for an oriented hypersurface of a Riemannian or Lorentzian group, of the support function in the direction of 
an element of its Lie algebra; in particular, we reobtain Theorem $1$ of \cite{Ripoll:03}, and extend it to the context of Lorentzian groups.
Finally, Section \ref{section:CMC hypersurfaces of Riemannian groups} deals with the case of cmc hypersurfaces of a 
Riemannian group, and Section \ref{section:CMC hypersurfaces of Lorentzian groups} treats the case of cmc spacelike hypersurfaces of
Lorentzian groups. Throughout, we also discuss prototype examples of the described phenomena.

\section{Semi-Riemannian groups}\label{section:Riemannian groups}

In this section, we recall some basic facts on Lie groups and Lie algebras which will be needed in the sequel, as well as fix some notations
and conventions to which we will stick to along the rest of the paper. We refer the reader to \cite{doCarmo:92},
\cite{Humphreys:97}, \cite{Knapp:02}, \cite{Lee:02} and \cite{O'Neill:83} for further discussions on the topics and results collected here.

Given an $(n+1)-$dimensional Lie group $G$, we let $\mathfrak g$ denote the corresponding Lie algebra of left invariant vector fields. For a
fixed basis $\mathcal B=(X_1,\ldots,X_{n+1})$ for $\mathfrak g$, the left invariance of $[X_i,X_j]$ assures the existence of real
numbers $c_{ij}^k$ such that
$$[X_i,X_j]=\sum_{k=1}^{n+1}c_{ij}^kX_k;$$
these numbers are said to be the structure constants of $\mathfrak g$ with respect to the basis $\mathcal B$, and clearly satisfy the relations 
$c_{ij}^k=-c_{ji}^k$, for all $1\leq i,j,k\leq n+1$; in particular, $c_{ii}^j=0$, for all such $i,j$.

We assume that $G$ is a connected Riemannian or Lorentzian group, i.e., that it is furnished with a bi-invariant metric tensor $\langle\cdot,\cdot\rangle$,
with index $\nu=0$ or $1$, respectively. 

For $X\in\mathfrak X(G)\setminus\{0\}$, we define the sign $\epsilon_X$ of $X$ as 
$$\epsilon_X=\text{sign}(\langle X,X\rangle)=\frac{\langle X,X\rangle}{|\langle X,X\rangle|}.$$
Therefore, the norm $|X|$ of $X$ is given by
$$|X|=\sqrt{\epsilon_X\langle X,X\rangle}.$$

Application of the Gramm-Schmidt algorithm to $\mathfrak g$ allows us to take the basis $\mathcal B$ of $\mathfrak g$ to be orthonormal, i.e., 
such that
$$\langle X_i,X_j\rangle=\epsilon_j\delta_{ij},$$
for $1\leq i,j\leq n+1$, where $\epsilon_j=\epsilon_{X_j}$ for all such $j$.

The bi-invariance of $\langle\cdot,\cdot\rangle$ gives us Weyl's relation
\begin{equation}\label{eq:Weyl's relation}
 \langle[X,Y],Z\rangle=\langle X,[Y,Z]\rangle,
\end{equation}
for all $X,Y,Z\in\mathfrak g$. By making the substitutions $X=X_i$, $Y=X_j$ and $Z=X_k$ at it, we get the equalities
\begin{equation}\label{eq:structure constants}
c_{ij}^k\epsilon_k=c_{jk}^i\epsilon_i,
\end{equation}
for all $1\leq i,j,k\leq n+1$; these, in turn, imply
$$c_{ij}^i\epsilon_i=-c_{ji}^i\epsilon_i=-c_{ii}^j\epsilon_j=0$$
and, hence,
\begin{equation}\label{eq:structure constants again}
c_{ij}^i=c_{ii}^j=0,
\end{equation}
for all $1\leq i,j\leq n+1$.

If $\tilde\nabla$ denotes the Levi-Civita connection of $G$ with respect to $\langle\cdot,\cdot\rangle$, it is a standard fact that
%as an easy consequence of (\ref{eq:Weyl's relation}), we get the relation
\begin{equation}\label{eq:Levi-Civita connection I}
\tilde\nabla_XY=\frac{1}{2}[X,Y],
\end{equation}
for all $X,Y\in\mathfrak g$. 

Back to the orthonormal basis $\mathcal B$, take a general $V\in\mathfrak X(G)$, say $V=\sum_{i=1}^{n+1}\alpha_iX_i$. Then,
$\alpha_i=\epsilon_i\langle V,X_i\rangle$ and, from the above,
\begin{equation}\label{eq:Levi-Civita connection II}
\tilde\nabla_VX_j=\sum_{i=1}^{n+1}\alpha_i\tilde\nabla_{X_i}X_j=\frac{1}{2}\sum_{i,k=1}^{n+1}\epsilon_i\langle V,X_i\rangle c_{ij}^kX_k.
\end{equation}

Concerning curvature, if $R_G:\mathfrak X(G)\times\mathfrak X(G)\times\mathfrak X(G)\rightarrow\mathfrak X(G)$ stands for the curvature tensor
of $G$, it is also standard that
\begin{equation}\label{eq:curvature tensor}
R_G(X,Y)Z=-\frac{1}{4}[[X,Y],Z],
\end{equation}
for all $X,Y,Z\in\mathfrak g$. If $X$ and $Y$ are orthonormal and span a nondegenerate $2-$plane in $\mathfrak g$, let $K_G(X,Y)$ 
denote the corresponding sectional curvature of $G$; then, it follows easily from (\ref{eq:Weyl's relation}) and (\ref{eq:curvature tensor}) 
that
\begin{equation}\label{eq:sectional curvature}
K_G(X,Y)=\frac{1}{4}\epsilon_X\epsilon_Y\epsilon_{[X,Y]}|[X,Y]|^2.
\end{equation}
%In particular, if $G$ is Riemannian, then it has nonnegative sectional curvature. 

For a Lie group $G$, the adjoint representation of its Lie algebra $\mathfrak g$ is the liner map
$\text{ad}:\mathfrak g\to\text{End}(\mathfrak g)$, such that
$$\text{ad}(X)(Y)=[X,Y].$$
The Killing form of $\mathfrak g$ is the symmetric bilinear form $B:\mathfrak g\times\mathfrak g\to\mathbb R$, such that
$$B(X,Y)=\text{tr}(\text{ad}(X)\text{ad}(Y));$$
it is well known that it satisfies
%The properties of the trace operation of Endomorphisms, together with the fact that adjoint representation of $\mathfrak g$ is a Lie derivation, give
\begin{equation}\label{eq:Weyl's relation for Killing forms}
B([X,Y],Z)=B(X,[Y,Z]),
\end{equation}
for all $X,Y,Z\in\mathfrak g$. If $G^{n+1}$ is furnished with a bi-invariant metric tensor, then, an easy computation by using 
(\ref{eq:curvature tensor}) and an orthonormal basis $\mathcal B$ for $\mathfrak g$ gives
\begin{equation}\label{eq:Ricci tensor and Killing form}
\text{Ric}_G(X,Y)=-\frac{1}{4}B(X,Y),
\end{equation}
for all $X,Y\in\mathfrak g$.

The canonical metric tensors on the Euclidean space $\mathbb R^{n+1}$ and the Lorentz-Minkowski space $\mathbb L^{n+1}$ are obviously 
bi-invariant. On the other hand, as was pointed out at Section \ref{section:Introduction}, a theorem of H. Weyl assures that every 
compact connected Lie group can be turned into a Riemannian group.

A little more into the algebraic side, we will get another interesting class of examples of semi-Riemannian groups, due to the following 
discussion.

Given a Lie algebra $\mathfrak g$, recall that its center $\mathcal Z(\mathfrak g)$ is the set of all $X\in\mathfrak g$ such that $[X,Y]=0$, 
for all $Y\in\mathfrak g$; an ideal of $\mathfrak g$ is a Lie subalgebra $\mathfrak h$ of $\mathfrak g$ such that $[X,Y]\in\mathfrak h$, 
for all $X\in\mathfrak h$, $Y\in\mathfrak g$. A Lie algebra $\mathfrak g$ is simple if it is nonabelian and its only ideals are $\mathfrak g$ and $\{0\}$, and semisimple if it 
is equal to a direct sum of simple subalgebras. The classical families $A_n$, $B_n$, $C_n$ and $D_n$ of matricial Lie algebras 
(cf. Chapter $1$ of \cite{Humphreys:97}) are all examples of finite dimensional semisimple Lie algebras. It is a standard fact of Lie 
algebra theory that all finite dimensional semisimple Lie algebras (and not only the simple ones) have trivial center.
A Lie group $G$ is semisimple if its Lie algebra is semisimple. On the other hand, for a given finite dimensional semisimple Lie algebra 
$\mathfrak g$, classical Lie group theory (cf. Theorem $20.20$ of \cite{Lee:02}) assures the existence of a
connected Lie group $G$ whose Lie algebra is isomorphic to $\mathfrak g$. Therefore, there are plenty of examples of semisimple Lie groups.

Our interest in semisimple Lie groups comes from Cartan's criterion for semisimplicity (cf. Theorem $1.45$ of \cite{Knapp:02}), which states 
that a Lie group $G$ is semisimple if, and only if, it has nondegenerate Killing form. Therefore, if $G$ is semisimple, it follows from
(\ref{eq:Weyl's relation for Killing forms}) and Proposition $11.9$ of \cite{O'Neill:83} that its Killing form $B$ turns $G$ into a 
semi-Riemannian group. Moreover, thanks to (\ref{eq:Ricci tensor and Killing form}), such a semi-Riemannian group is an Einstein manifold.

The following algebraic consequence of the above discussion is straightforward and will be needed at Section 
\ref{section:CMC hypersurfaces of Riemannian groups}.

\begin{lemma}\label{lemma:algebraic lemma}
 If $G^{n+1}$ is a Riemannian group with Lie algebra $\mathfrak g$, then $\mathcal Z(\mathfrak g)\neq\{0\}$ if, and only if, $G$ admits a 
connected Lie subgroup of dimension $n$.
\end{lemma}

%\begin{proof}
% We prove the ``only if'' implication, the proof of the other one being totally analogous. 
%Let $X\in\mathfrak g$ be a nonzero element of $\mathcal Z(\mathfrak g)$, and $X_1,\ldots,X_n\in\mathfrak g$ be such that $(X_1,\ldots,X_n,X)$ 
%is an orthonormal basis of $\mathfrak g$. Then, for $1\leq i,j\leq n$, Weyl's relation gives
%$$\langle[X_i,X_j],X\rangle=\langle X_i,[X_j,X]\rangle=\langle X_i,0\rangle=0$$
%and, hence, the subspace $\mathfrak h$ of $\mathfrak g$ generated by $X_1,\ldots,X_n$ is a Lie subalgebra of $\mathfrak g$. By a fundamental 
%theorem of Lie (cf. Theorem $20.13$ of \cite{Lee:02}), there exists a unique connected Lie subgroup $H$ of $G$ whose Lie algebra is isomorphic to $\mathfrak h$; 
%thus, $H$ is $n-$dimensional.
%\end{proof}

\section{Support functions on hypersurfaces}\label{section:Support functions on hypersurfaces}

Along this section, we will stick to the following notations: we will consider a Riemannian or Lorentzian group $G^{n+1}$ and a Riemannian 
immersion $\varphi:M^n\rightarrow G^{n+1}$, i.e., an immersion of an $n-$dimensional orientable differentiable manifold $M^n$ into $G^{n+1}$, 
such that $M$ turns into a Riemannian manifold when furnished with the induced metric. (In the case where $G$ is Lorentzian, one also calls $M$ 
a spacelike hypersurface of $G$.) We always assume that $M$ is furnished with the induced metric, and from now on we will call $M$ simply
a hypersurface of $G$. 

We orient $M$ by the choice of a unit normal vector field $N$, so that $\epsilon_N=(-1)^{\nu}$. We let $A(\cdot)=-\tilde\nabla_{(\cdot)}N$ 
denote the corresponding shape operator, and $H$ the mean curvature of $\varphi$ with respect to $N$, so that
$$H=\epsilon_N\text{tr}(A).$$ 

For $X\in\mathfrak g$, we let $f_X:M^n\rightarrow\mathbb R$ denote the support function of $M$ in the direction of $X$, i.e.,
$$f_X(p)=\langle N,X\rangle_p,$$
for every $p\in M$. Given an orthonormal basis $\mathcal B=(X_1,\ldots,X_{n+1})$ for $\mathfrak g$, and $1\leq j\leq n+1$, we let 
$f_j=f_{X_j}$.

We begin with a symmetry relation on the structure constants of $\mathcal B$ and, in the sequel, compute the gradient and the Laplacian of 
$f_j$.

\begin{lemma}\label{lemma:relation on structure constants}
$\sum_{i,j=1}^{n+1}c_{jl}^i\epsilon_jf_if_j=0$, for $1\leq l\leq n+1$.
\end{lemma}

\begin{proof}
It suffices to change indices and invoke (\ref{eq:structure constants}):
$$\sum_{i,j=1}^{n+1}c_{jl}^i\epsilon_jf_if_j=\sum_{i,j=1}^{n+1}c_{il}^j\epsilon_if_jf_i=-\sum_{i,j=1}^{n+1}c_{jl}^i\epsilon_i^2\epsilon_jf_if_j=-\sum_{i,j=1}^{n+1}c_{jl}^i\epsilon_jf_if_j.$$
\end{proof}

For what comes next, for a section $X$ of $TG_{|M}$, we let $X^{\top}$ denote the orthogonal projection of $X$ onto $TM$.

\begin{lemma}\label{lemma:gradient of support functions}
$\nabla f_j=-AX_j^{\top}+\frac{1}{2}\sum_{i,l=1}^{n+1}c_{lj}^i\epsilon_lf_iX_l^{\top}$, for $1\leq j\leq n+1$.
\end{lemma}

\begin{proof}
Fix $p\in M$ and let $(e_1,\ldots,e_n)$ be an orthonormal frame field on a neighborhood of $p$ in $M$. Then, for $1\leq k\leq n$ and thanks to
(\ref{eq:Levi-Civita connection II}), we have
\begin{equation}\label{eq:derivative of support function at a fixed direction}
 \begin{split}
 e_k(f_j)&\,=\langle\tilde\nabla_{e_k}N,X_j\rangle+\langle N,\tilde\nabla_{e_k}X_j\rangle\\
&\,=-\langle Ae_k,X_j\rangle+\frac{1}{2}\sum_{i,l=1}^{n+1}\langle N,\epsilon_l\langle e_k,X_l\rangle c_{lj}^iX_i\rangle\\
&\,=-\langle Ae_k,X_j\rangle+\frac{1}{2}\sum_{i,l=1}^{n+1}\epsilon_l\langle e_k,X_l\rangle c_{lj}^if_i.
 \end{split}
\end{equation}
Hence,
\begin{equation}\nonumber
\begin{split}
 \nabla f_j&\,=\sum_{k=1}^ne_k\langle N,X_j\rangle e_k\\
&\,=-\sum_{k=1}^n\langle e_k,AX_j^{\top}\rangle e_k+\frac{1}{2}\sum_{k=1}^n\sum_{i,l=1}^{n+1}\epsilon_l\langle e_k,X_l\rangle c_{lj}^if_ie_k\\
&\,=-AX_j^{\top}+\frac{1}{2}\sum_{i,l=1}^{n+1}\epsilon_lc_{lj}^if_iX_l^{\top}.
 \end{split}
\end{equation}
\end{proof}

We shall also need the following expression for the squared norm of the gradient of $f_j$.

\begin{lemma}\label{lemma:square norm of the gradient}
\begin{equation}\nonumber
\begin{split} 
|\nabla f_j|^2&\,=|AX_j^{\top}|^2-\sum_{i,l=1}^{n+1}\langle X_l^{\top},AX_j^{\top}\rangle c_{lj}^i\epsilon_lf_i\\
&\,\,\,\,\,\,\,+\frac{1}{4}\sum_{i,l,r=1}^{n+1}c_{lj}^ic_{lj}^r\epsilon_lf_if_r.
 \end{split}
\end{equation}
\end{lemma}

\begin{proof}
This is a straightforward calculation. From the previous lemma, we get
\begin{equation}\nonumber
 \begin{split}
|\nabla f_j|^2&\,=\langle-AX_j^{\top}+\frac{1}{2}\sum_{i,l=1}^{n+1}c_{lj}^i\epsilon_lf_iX_l^{\top},-AX_j^{\top}+\frac{1}{2}\sum_{r,s=1}^{n+1}c_{sj}^r\epsilon_sf_rX_s^{\top}\rangle\\
&\,=|AX_j^{\top}|^2-\frac{1}{2}\sum_{r,s=1}^{n+1}\langle AX_j^{\top},X_s^{\top}\rangle c_{sj}^r\epsilon_sf_r-\frac{1}{2}\sum_{i,l=1}^{n+1}\langle X_l^{\top},AX_j^{\top}\rangle c_{lj}^i\epsilon_lf_i\\
&\,\,\,\,\,\,\,+\frac{1}{4}\sum_{i,l,r,s=1}^{n+1}c_{lj}^ic_{sj}^r\epsilon_l\epsilon_sf_if_r(\underbrace{\langle X_l,X_s\rangle}_{\delta_{ls}\epsilon_s}-\epsilon_Nf_lf_s)\\
&\,=|AX_j^{\top}|^2-\sum_{i,l=1}^{n+1}\langle X_l^{\top},AX_j^{\top}\rangle c_{lj}^i\epsilon_lf_i+\frac{1}{4}\sum_{i,l,r=1}^{n+1}c_{lj}^ic_{lj}^r\epsilon_lf_if_r\\
&\,\,\,\,\,\,\,-\frac{\epsilon_N}{4}\sum_{i,l,r,s=1}^{n+1}c_{lj}^ic_{sj}^r\epsilon_l\epsilon_sf_if_lf_rf_s\\
&\,=|AX_j^{\top}|^2-\sum_{i,l=1}^{n+1}\langle X_l^{\top},AX_j^{\top}\rangle c_{lj}^i\epsilon_lf_i+\frac{1}{4}\sum_{i,l,r=1}^{n+1}c_{lj}^ic_{lj}^r\epsilon_lf_if_r,
 \end{split}
\end{equation}
for, by Lemma \ref{lemma:relation on structure constants},
$$\sum_{i,l,r,s=1}^{n+1}c_{lj}^ic_{sj}^r\epsilon_l\epsilon_sf_if_lf_rf_s=\left(\sum_{i,l=1}^{n+1}c_{lj}^i\epsilon_lf_if_l\right)^2=0.$$
\end{proof}

The following lemma extends Theorem $1$ of \cite{Ripoll:03} to Lorentzian groups. For the sake of completeness, and since there is no gain 
in presenting its proof only in the Lorentzian setting, we present a (different) proof of it, which encompasses both the Riemannian and 
the Lorentzian cases.

\begin{lemma}\label{lemma:Laplacian of support functions}
If $\Delta$ stands for the Laplacian of $M$, then
$$\Delta f_j=-\epsilon_NX_j^{\top}(H)-\epsilon_N(|A|^2+\text{\rm Ric}_G(N))f_j,$$
where $\text{\rm Ric}_G(N)$ denotes the Ricci curvature of $G$ in the direction of $N$.
\end{lemma}

\begin{proof}
 Let $\nabla$ be the Levi-Civita connection of $M$, $\tilde\nabla$ that of $G$ and $\alpha$ the second fundamental form of $\varphi$.
For a fixed $p\in M$, let $(e_1,\ldots,e_n)$ be an orthonormal frame field on a neighborhood of $p$ in $M$, geodesic at $p$. Then,
computing at $p$ with the help of (\ref{eq:derivative of support function at a fixed direction}), we successively get
\begin{equation}\nonumber
 \begin{split}
\Delta f_j&=\,\sum_{k=1}^ne_k(e_k(f_j))=-\sum_{k=1}^ne_k\langle Ae_k,X_j\rangle+\frac{1}{2}\sum_{k=1}^n\sum_{i,l=1}^{n+1}e_k\langle e_k,X_l\rangle c_{lj}^i\epsilon_lf_i\\
&\,\,\,\,\,\,\,+\frac{1}{2}\sum_{k=1}^n\sum_{i,l=1}^{n+1}\langle e_k,X_l\rangle c_{lj}^i\epsilon_le_k(f_i)\\
&\,=-\underbrace{\sum_{k=1}^n\langle\nabla_{e_k}Ae_k,X_j^{\top}\rangle}_{(I)}-\underbrace{\sum_{k=1}^n\langle Ae_k,\nabla_{e_k}X_j^{\top}\rangle}_{(II)}\\
&\,\,\,\,\,\,\,+\underbrace{\frac{1}{2}\sum_{k=1}^n\sum_{i,l=1}^{n+1}\langle\alpha(e_k,e_k),X_l\rangle c_{lj}^i\epsilon_lf_i}_{(III)}+\underbrace{\frac{1}{2}\sum_{k=1}^n\sum_{i,l=1}^{n+1}\langle e_k,\tilde\nabla_{e_k}X_l\rangle c_{lj}^i\epsilon_lf_i}_{(IV)}\\
&\,\,\,\,\,\,\,+\underbrace{\frac{1}{2}\sum_{k=1}^n\sum_{i,l=1}^{n+1}\langle e_k,X_l\rangle c_{lj}^i\epsilon_le_k(f_i)}_{(V)}.
 \end{split}
\end{equation}
We compute separately each of the summands $(I)$ to $(V)$ above at $p$.

For $(I)$, it follows from the self adjointness of $\nabla_{e_k}A$ and Codazzi's equation that
\begin{equation}\nonumber
 \begin{split}
(I)&\,=\sum_{k=1}^n\langle(\nabla_{e_k}A)e_k,X_j^{\top}\rangle=\sum_{k=1}^n\langle(\nabla_{e_k}A)X_j^{\top},e_k\rangle\\
&\,=\sum_{k=1}^n\langle(R_G(X_j^{\top},e_k)N)^{\top},e_k\rangle+\sum_{k=1}^n\langle(\nabla_{X_j^{\top}}A)e_k,e_k\rangle\\
&\,=\sum_{k=1}^n\langle R_G(X_j,e_k)N,e_k\rangle-\sum_{k=1}^n\langle R_G(\epsilon_Nf_jN,e_k)N,e_k\rangle\\
&\,\,\,\,\,\,\,+\sum_{k=1}^n\langle\nabla_{X_j^{\top}}Ae_k,e_k\rangle\\
&\,=-\text{Ric}_G(X_j,N)+\epsilon_Nf_j\text{Ric}_G(N,N)+X_j^{\top}\left(\sum_{k=1}^n\langle Ae_k,e_k\rangle\right)\\
&\,=-\text{Ric}_G(X_j,N)+\epsilon_N\text{Ric}_G(N,N)f_j+\epsilon_NX_j^{\top}(H).
 \end{split}
\end{equation}

The computation of $(II)$ is straightforward and gives
\begin{equation}\nonumber
 \begin{split}
(II)&\,=\sum_{k=1}^n\langle Ae_k,\tilde\nabla_{e_k}X_j^{\top}\rangle=\sum_{k=1}^n\langle Ae_k,\tilde\nabla_{e_k}(X_j-\epsilon_Nf_jN)\rangle\\
&\,=\sum_{k=1}^n\langle Ae_k,\tilde\nabla_{e_k}X_j\rangle+\epsilon_Nf_j\sum_{k=1}^n\langle Ae_k,-\tilde\nabla_{e_k}N\rangle\\
&\,=\sum_{k=1}^n\langle Ae_k,\tilde\nabla_{e_k}X_j\rangle+\epsilon_Nf_j|A|^2\\
&\,=\frac{1}{2}\sum_{k=1}^n\sum_{i,l=1}^{n+1}\langle Ae_k,X_l\rangle\langle e_k,X_i\rangle c_{ij}^l\epsilon_i+\epsilon_Nf_j|A|^2\\
&\,=\frac{1}{2}\sum_{i,l=1}^{n+1}\langle AX_l^{\top},X_i^{\top}\rangle c_{ij}^l\epsilon_i+\epsilon_N|A|^2f_j.\\
 \end{split}
\end{equation}

Likewise for $(III)$, it follows from Lemma \ref{lemma:relation on structure constants} that
\begin{equation}\nonumber
 \begin{split}
(III)&\,=\frac{1}{2}\sum_{k=1}^n\sum_{i,l=1}^{n+1}\langle\alpha(e_k,e_k),X_l\rangle c_{lj}^i\epsilon_lf_i=\frac{1}{2}\sum_{i,l=1}^{n+1}\langle HN,X_l\rangle c_{lj}^i\epsilon_lf_i\\
&\,=\frac{H}{2}\sum_{i,l=1}^{n+1}c_{lj}^i\epsilon_lf_if_l=0.
 \end{split}
\end{equation}

We get rid of $(IV)$ by using again (\ref{eq:structure constants again}) and Lemma \ref{lemma:relation on structure constants}:
\begin{equation}\nonumber
 \begin{split}
(IV)&\,=\frac{1}{4}\sum_{k=1}^n\sum_{i,l,r,s=1}^{n+1}\langle e_k,X_r\rangle\langle e_k,X_s\rangle c_{sl}^r\epsilon_sc_{lj}^i\epsilon_lf_i\\
&\,=\frac{1}{4}\sum_{i,l,r,s=1}^{n+1}(\underbrace{\langle X_r,X_s\rangle}_{\delta_{rs}\epsilon_s}-\epsilon_Nf_rf_s)c_{sl}^r\epsilon_sc_{lj}^i\epsilon_lf_i\\
&\,=\frac{1}{4}\sum_{i,l,r=1}^{n+1}c_{rl}^rc_{lj}^i\epsilon_lf_i-\frac{\epsilon_N}{4}\sum_{i,l,r,s=1}^{n+1}c_{sl}^r\epsilon_sc_{lj}^i\epsilon_lf_if_rf_s=0.
 \end{split}
\end{equation}

Finally, we compute $(V)$ by substituting (\ref{eq:derivative of support function at a fixed direction}) into it:
\begin{equation}\nonumber
 \begin{split}
(V)&\,=\frac{1}{2}\sum_{k=1}^n\sum_{i,l=1}^{n+1}\langle e_k,X_l\rangle c_{lj}^i\epsilon_l\left(-\langle Ae_k,X_i\rangle+\frac{1}{2}\sum_{r,s=1}^{n+1}\langle e_k,X_r\rangle c_{ri}^s\epsilon_rf_s\right)\\
&\,=-\frac{1}{2}\sum_{k=1}^n\sum_{i,l=1}^{n+1}\langle e_k,X_l^{\top}\rangle c_{lj}^i\epsilon_l\langle e_k,AX_i^{\top}\rangle\\
&\,\,\,\,\,\,\,+\frac{1}{4}\sum_{k=1}^n\sum_{i,l,r,s=1}^{n+1}\langle e_k,X_l\rangle\langle e_k,X_r\rangle c_{lj}^i\epsilon_lc_{ri}^s\epsilon_rf_s\\
&\,=-\frac{1}{2}\sum_{i,l=1}^{n+1}\langle X_l^{\top},AX_i^{\top}\rangle c_{lj}^i\epsilon_l+\frac{1}{4}\sum_{i,l,r,s=1}^{n+1}(\underbrace{\langle X_l,X_r\rangle}_{\delta_{lr}\epsilon_r}-\epsilon_Nf_lf_r)c_{lj}^i\epsilon_lc_{ri}^s\epsilon_rf_s\\
&\,=-\frac{1}{2}\sum_{i,l=1}^{n+1}\langle AX_i^{\top},X_l^{\top}\rangle c_{lj}^i\epsilon_l+\frac{1}{4}\sum_{i,l,s=1}^{n+1}c_{lj}^ic_{li}^s\epsilon_lf_s\\
&\,\,\,\,\,\,\,-\frac{\epsilon_N}{4}\sum_{i,l,r,s=1}^{n+1}c_{lj}^i\epsilon_lc_{ri}^s\epsilon_rf_lf_rf_s.
 \end{split}
\end{equation}
The last summand above equals $0$, again due to Lemma \ref{lemma:relation on structure constants}; to the middle one we have,
thanks to (\ref{eq:Weyl's relation}) and (\ref{eq:structure constants}),
\begin{equation}\nonumber
 \begin{split}
\sum_{i,l,s=1}^{n+1}c_{lj}^ic_{li}^s\epsilon_lf_s&\,=-\sum_{i,l,s=1}^{n+1}c_{lj}^ic_{ls}^i\epsilon_i\epsilon_l\epsilon_sf_s\\
&\,=-\sum_{i,l,s=1}^{n+1}\langle[X_l,X_j],X_i\rangle\langle[X_l,X_s],X_i\rangle\epsilon_i\epsilon_l\epsilon_sf_s\\
&\,=\sum_{l,s=1}^{n+1}\langle[X_j,X_l],[X_l,X_s]\rangle\epsilon_l\epsilon_sf_s\\
&\,=\sum_{l,s=1}^{n+1}\langle[[X_j,X_l],X_l],X_s\rangle\epsilon_l\epsilon_sf_s\\
&\,-4\sum_{l=1}^{n+1}\langle R_G(X_j,X_l),X_l,N\rangle\epsilon_l=-4\,\text{Ric}_G(X_j,N),
 \end{split}
\end{equation}
so that
$$(V)=-\frac{1}{2}\sum_{i,l=1}^{n+1}\langle AX_i^{\top},X_l^{\top}\rangle c_{lj}^i\epsilon_l-\text{Ric}_G(X_j,N).$$

If we now gather together the results from $(I)$ to $(V)$, we obtain
\begin{equation}\nonumber
 \begin{split}
\Delta f_j&\,=\text{Ric}_G(X_j,N)-\epsilon_N\text{Ric}_G(N,N)f_j-\epsilon_NX_j^{\top}(H)\\
&\,\,\,\,\,\,\,-\frac{1}{2}\sum_{i,l=1}^{n+1}\langle AX_l^{\top},X_i^{\top}\rangle c_{ij}^l\epsilon_i-\epsilon_N|A|^2f_j\\
&\,\,\,\,\,\,\,-\frac{1}{2}\sum_{i,l=1}^{n+1}\langle AX_i^{\top},X_l^{\top}\rangle c_{lj}^i\epsilon_l-\text{Ric}_G(X_j,N)\\
&\,=-\epsilon_NX_j^{\top}(H)-\epsilon_N(|A|^2+\text{Ric}_G(N,N))f_j,
\end{split}
\end{equation}
since $A$ is self adjoint and, by (\ref{eq:structure constants again}), $c_{ij}^l\epsilon_i+c_{lj}^i\epsilon_l=0$.
\end{proof}

With the previous computations at our disposal, we now proceed to analyse the geometry of cmc Riemannian hypersurfaces of Riemannian or Lorentzian 
groups. To this end, let $G^{n+1}$ be a Riemannian or Lorentzian group with Lie algebra $\mathfrak g$, and $\varphi:M^n\rightarrow G^{n+1}$ 
be a (Riemannian) hypersurface of $G$. For a given $X\in\mathfrak g$, we recall that $\varphi$ is {\em transversal} to $X$ if, for {\em every} 
$p\in M$, we have
$$T_{\varphi(p)}G=(\varphi_*)_pT_pM\oplus \mathbb RX_p.$$
If $M$ is connected and oriented by the unit normal vector field $N$, this is equivalent to the fact that the support function $f_X$ has a 
strict sign on $M$.

\section{CMC hypersurfaces of Riemannian groups}\label{section:CMC hypersurfaces of Riemannian groups}

Along all of this section, $G^{n+1}$ is an $(n+1)-$dimensional Riemannian group and $\varphi:M^n\to G^{n+1}$ a connected hypersurface of $G$, 
oriented by the choice of a globally defined unit normal vector field $N$. We also let $A$ denote the corresponding shape operator.

\begin{theorem}\label{thm:main theorem}
Let $G^{n+1}$ be a Riemannian group and $\varphi:M^n\rightarrow G^{n+1}$ be a compact connected oriented cmc hypersurface of $G$. If 
$\varphi$ is transversal to some element of $\mathfrak g$, then the following hold:
\begin{enumerate}
 \item[$(a)$] The Ricci curvature of $G$ in the direction of $N$ vanishes.
 \item[$(b)$] $\varphi$ is totally geodesic.
 \item[$(c)$] $\varphi(M)$ is a lateral class of an embedded Lie subgroup of $G$.
\end{enumerate}
\end{theorem}

\begin{proof}
If $M$ is transversal to $X\in\mathfrak g$, then $f_X$ has a strict sign on $M$. 
Since $\varphi$ is of constant mean curvature, it follows from Lemma \ref{lemma:Laplacian of support functions} that
\begin{equation}\label{eq:auxiliar para a prova}
\Delta f_X=-(|A|^2+\text{\rm Ric}_G(N))f_X.
\end{equation}

Now, relation (\ref{eq:sectional curvature}) guarantees that $\text{\rm Ric}_G(N)\geq 0$; therefore, since $f_X$ has a strict sign, 
it follows from (\ref{eq:auxiliar para a prova}) that $f_X$ is a superharmonic function on $M$. However, since $M$ 
is compact, Hopf's theorem gives $f_X$ to be constant on $M$. Pluging this information back on (\ref{eq:auxiliar para a prova}),
we arrive at $|A|^2+\text{\rm Ric}_G(N)=0$ on $M$, from which we conclude that $\varphi$ is totally geodesic and $\text{\rm Ric}_G(N)=0$.

To what is left to prove, as usual we use $\varphi$ to locally identify $M$ with its image in $G$. Also, since left translations are 
isometries of $G$, in order to prove that $M$ is a lateral class of a Lie subgroup of $G$ it suffices to assume that $M$ is totally geodesic
and such that $e\in M$, thus proving that $M$ is a Lie subgroup of $G$. Under these assumptions, 
%Suppose, then, that $M$ is totally geodesic and $e\in M$. The compactness of $M$ implies completeness; since it is also connected, its 
%Riemannian exponential map at $e$ is surjective. 
choose $v\in T_eM$ and let $\gamma_v:\mathbb R\rightarrow M$ be the 
maximal geodesic of $M$ that departs from $e$ with velocity $v$. The total geodesic character of $\varphi$ assures that 
$\gamma_v$ is also a maximal geodesic of $G$ and, thus, the bi-invariance of the metric forces $\gamma_v$ to be a $1-$parameter subgroup 
of $G$. We now choose a basis 
$(v_1,\ldots,v_n)$ for $T_eM$ and let $E_i\in\mathfrak g$ be such that $(E_i)_e=v_i$, for $1\leq i\leq n$. Since our previous
reasoning is true for all $v\in T_eM$, we arrive at the conclusion that the restriction of $E_i$ to $M$ is an element of $\mathfrak X(M)$, 
for $1\leq i\leq n$. Hence, $[E_i,E_j]_{|M}\in\mathfrak X(M)$ as well, for all $1\leq i,j\leq n$, and $E_1,\ldots,E_n$ generate a Lie subalgebra
of $\mathfrak g$. 

Let $\mathfrak h$ be the Lie subalgebra of $\mathfrak g$ generated by $E_1,\ldots,E_n$. If $H$ is the unique connected Lie subgroup of $G$
whose Lie algebra is canonically isomorphic to $\mathfrak h$, then (cf. Corollary $20.14$ of \cite{Lee:02}) $H$ is the union of the images 
of all $1-$parameter subgroups of $G$ generated by elements of $\mathfrak h$, so that $H=\varphi(M)$. If we now remind that $M$ is compact, 
hence closed, we may invoke Cartan's closed subgroup theorem to conclude that $M$ is an embedded Lie subgroup of $G$. 
\end{proof}

We now recall (cf. \cite{Ripoll:03}) the extension, to Riemannian groups, of the Gauss map $\eta:M^n\rightarrow\mathbb S^n(T_eG)$ of $M$ 
with respect to $N$: it is given, at $p\in M$, by $\eta(p)=Y_e$, where $Y$ is the unique element of $\mathfrak g$ satisfying $Y_p=N_p$, i.e.,
\begin{equation}\label{eq:expression for the Gauss map}
\eta(p)=((L_{p^{-1}})_*)_pN_p.
\end{equation}

For fixed $p\in M$ and $v\in T_pM$, let $\gamma:(-\epsilon,\epsilon)\rightarrow M$ be a smooth curve such that $\gamma(0)=p$, $\gamma'(0)=v$. 
Then,
\begin{equation}\nonumber
\begin{split}
 (\eta_p)_*(v)&\,=\frac{d}{dt}(\eta\circ\gamma)(t)\Big|_{t=0}=\frac{d}{dt}((L_{\gamma(t)^{-1}})_*)_{\gamma(t)}N_{\gamma(t)}\Big|_{t=0}\\
&\,=((L_{p^{-1}})_*)_p\frac{DN}{dt}\Big|_{t=0}=-((L_{p^{-1}})_*)_pA_pv
\end{split}
\end{equation}
and, hence,
\begin{equation}\label{eq:differential of the Gauss map}
(\eta_p)_*=-((L_{p^{-1}})_*)_pA.
\end{equation}

Finally, we recall that the nullity of $\varphi$ at $p$ is the nullity (in the linear algebraic sense) of $A$ at $p$.
The preceding computations, then, give the following interesting consequence of Theorem \ref{thm:main theorem}.

\begin{theorem}\label{thm:semisimple}
 Let $G^{n+1}$ be a semisimple Riemannian group and $\varphi:M^n\rightarrow G^{n+1}$ be a compact connected oriented cmc hypersurface of $G$.
Then, $M$ has degenerate Gauss map and minimal nullity at least $1$ at every point.
\end{theorem}

\begin{proof}
 As we observed in Section \ref{section:Riemannian groups}, the Lie algebra of a semisimple Lie group has trivial center, so that, 
by Lemma \ref{lemma:algebraic lemma}, $G$ does not admit a connected Lie subgroup of dimension $n$. Therefore, by Theorem \ref{thm:main theorem}, $\varphi$ 
cannot be transversal to any element of $\mathfrak g$. 

Let $\eta:M^n\rightarrow\mathbb S^n(T_eG)$ be the Gauss map of $G$ and $(X_1,\ldots,X_{n+1})$ be a basis for $\mathfrak g$. 
The above argument, together with (\ref{eq:expression for the Gauss map}), says that, at every point $p$ of $M$, the vector 
$\eta(p)$ is orthogonal to at least one of $(X_1)_e$, \ldots, $(X_{n+1})_e$.
Therefore, the image of $\eta$ is contained in the union of $n+1$ great circles of $\mathbb S^n(T_eG)$ and, thus, has measure zero in 
$\mathbb S^n(T_eG)$. In particular, $\eta$ cannot have rank $n$ at any point of $M$.

Therefore, since $\eta_p$ has rank at most $n-1$ and $((L_{p^{-1}})_*)_p$ is an isomorphism, we conclude from 
(\ref{eq:differential of the Gauss map}) that $A_p$ has rank at most $n-1$ and, thus, nullity at least $1$.
\end{proof}

We now turn to the case in which $M$ is complete and noncompact, and begin by showing that, if $\varphi$ is minimal and unstable, then the 
previous theorem remains true, even if $G$ is not semisimple. To the benefit of understanding, the reader should contrast the theorem below 
to the case of a catenoid in $\mathbb R^3$: it is unstable, but it is not transversal to any element of the Lie algebra of $\mathbb R^3$.

\begin{theorem}\label{thm:stability}
Let $G^{n+1}$ be a Riemannian group and $\varphi:M^n\rightarrow G^{n+1}$ be a complete noncompact connected oriented minimal hypersurface of $G$. 
\begin{enumerate}
 \item[$(a)$] If $\varphi$ is transversal to some element $X$ of $\mathfrak g$, then $\varphi$ is stable.
 \item[$(b)$] If $\varphi$ is unstable, then $M$ has degenerate Gauss map and minimal relative nullity at least $1$ at every point.
\end{enumerate}
\end{theorem}

\begin{proof}
It is a standard fact (cf. Chapter 6 of \cite{Xin:03}) that the stability of $\varphi$ is equivalent to the positivity of the associated
Jacobi operator
$$L=\Delta+(\text{Ric}_G(N)+|A|^2).$$
On the other hand, by a theorem of Fischer-Colbrie and Schoen (cf. \cite{Schoen:80}), this is equivalent to the existence of a positive 
function $f:M\rightarrow\mathbb R$ such that $Lf=0$. In view of Lemma \ref{lemma:Laplacian of support functions}, this argument proves (a).

Concerning (b), if $\varphi$ is unstable, then it follows from (a) that it cannot be transversal to any element of the Lie algebra of $G$.
Therefore, an argument identical of that of the proof of the previous result finishes the proof.
\end{proof}

In order to generalize Theorem \ref{thm:main theorem} to the case of a complete and noncompact $M$, we shall need to ask a little more 
from our hypersurface. To properly state our hypotheses, we set $\pi_X:T_eG\rightarrow T_eG$ to be the orthogonal projection onto 
the orthogonal complement of $X_e$ in $T_eG$. Since left translations are isometries of $G$, it follows from 
(\ref{eq:expression for the Gauss map}) that
$$|\pi_X(\eta(p))|^2=|\eta(p)|^2-\langle\eta(p),X_e\rangle^2=|N_p|^2-\langle N_p,X_p\rangle^2,$$
i.e.,
\begin{equation}\label{eq:size of the orthogonal projection I}
|\pi_X(\eta)|^2=1-f_X^2.
\end{equation}
Hence, if $|X|=1$, then
\begin{equation}\label{eq:size of the orthogonal projection II}
|X^{\top}|^2=|X|^2-\langle X,N\rangle^2=1-f_X^2=|\pi_X(\eta)|^2.
\end{equation}

Our result is, thus, as follows.

\begin{theorem}\label{thm:second theorem}
Let $G^{n+1}$ be a Riemannian group and $\varphi:M^n\rightarrow G^{n+1}$ be a complete noncompact connected oriented cmc hypersurface of $G$, 
with shape operator $A$ and Gauss map $\eta$ relative to $N$. Assume that $A$ is bounded on $M$ and that $\varphi$ is transversal to some 
element $X$ of $\mathfrak g$. If $|\pi_X(\eta)|$ is integrable on $M$, then the following hold:
\begin{enumerate}
 \item[$(a)$] The Ricci curvature of $G$ in the direction of $N$ vanishes. 
 \item[$(b)$] $\varphi$ is totally geodesic.
 \item[$(c)$] If $\varphi$ is proper, then $\varphi(M)$ is a lateral class of an embedded Lie subgroup of $G$.
\end{enumerate}
\end{theorem}

\begin{proof}
As in the proof of Theorem \ref{thm:main theorem}, we conclude that $f_X$ is superharmonic.
If $|\nabla f_X|$ is integrable on $M$, then, by Yau's extension of Hopf's theorem to complete noncompact Riemannian manifolds 
(cf. \cite{Yau:76}), $f_X$ will be constant on $M$ and, as in the previous proof, we shall get items (a) and (b).

In order to prove that $|\nabla f_X|$ is integrable on $M$ assume, without loss of generality, that $X$ is a unit vector field. Also, 
from the very beginning, take an orthonormal basis $\mathcal B=(X_1,\ldots,X_{n+1})$ for $\mathfrak g$, such that $X_{n+1}=X$ and, thus,
$f_X=f_{n+1}$. By Lemma \ref{lemma:square norm of the gradient}, and taking into account that $c_{l,n+1}^{n+1}=0$ for $1\leq l\leq n+1$
(cf. (\ref{eq:structure constants again})), we get
\begin{equation}\nonumber
 \begin{split} 
|\nabla f_X|^2&\,=|AX^{\top}|^2-\sum_{i,l=1}^n\langle X_l^{\top},AX^{\top}\rangle c_{l,n+1}^if_i\\
&\,\,\,\,\,\,\,+\frac{1}{4}\sum_{i,l,r=1}^nc_{l,n+1}^ic_{l,n+1}^rf_if_r.
 \end{split}
\end{equation}
Since $|X_l^{\top}|\leq 1$ for $1\leq l\leq n+1$, we can write, with the help of Cauchy-Schwarz inequality,
\begin{equation}\nonumber
 \begin{split} 
|\nabla f_X|^2&\,\leq|A|^2|X^{\top}|^2+\sum_{i,l=1}^n|X_l^{\top}||A||X^{\top}||c_{l,n+1}^i||f_i|\\
&\,\,\,\,\,\,\,+\frac{1}{4}\sum_{i,l,r=1}^{n+1}|c_{l,n+1}^i||c_{l,n+1}^r||f_i||f_r|\\
&\,\leq|A|^2|X^{\top}|^2+C|A||X^{\top}|\sum_{i=1}^n|f_i|+\frac{C}{4}\sum_{i,l=1}^n|f_i||f_l|\\
&\,\leq C\left(|A|^2|X^{\top}|^2+|A||X^{\top}|\sum_{i=1}^n|f_i|+\frac{1}{4}\left(\sum_{i=1}^n|f_i|\right)^2\right)\\
&\,=C\left(|A||X^{\top}|+\frac{1}{2}\sum_{i=1}^n|f_i|\right)^2,
 \end{split}
\end{equation}
for some constant $C>0$. Hence, due to the boundedness of $|A|$, (\ref{eq:size of the orthogonal projection I}) and 
(\ref{eq:size of the orthogonal projection II}), we get (incorporating constants)
\begin{equation}\nonumber
 \begin{split} 
|\nabla f_X|&\,\leq C\left(|X^{\top}|+\frac{1}{2}\sum_{i=1}^n|f_i|\right)\leq C\left(|X^{\top}|+\sqrt{\sum_{i=1}^nf_i^2}\right)\\
&\,=C\left(|\pi_X(\eta)|+\sqrt{1-f_X^2}\right)=C|\pi_X(\eta)|.
\end{split}
\end{equation}

Now, assume further that $\varphi$ is proper. As in the previous proof, we get the existence of a Lie subgroup $H$ of $G$, 
such that $H=\varphi(M)$. However, since proper maps are closed, we are done by applying Cartan's closed subgroup theorem once again.
\end{proof}

\begin{remark}
{\em
As quoted at the Introduction, the prototype example of the situation covered by theorems \ref{thm:main theorem} and \ref{thm:second theorem} 
is that of the product Riemannian groups $G^{n+1}=K\times\mathbb R$ and $K\times\mathbb S^1$, where $\mathbb R$ and $\mathbb S^1$ have their 
canonical metrics and $K$ is a Riemannian group.
Moreover, due to de Rham's decomposition theorem, the Riemannian group $G$ of those results turns out to be, in a neighborhood of each point 
of $M$, locally isometric to a Riemannian product of pieces of $H$ and the $1-$parameter subgroup generated by $X$. 
However, this need not be so globally, and a typical example in the compact case is given by taking $G$ and $M$ to be respectively equal
to the unitary group $U(m)$ and the special unitary group $SU(m)$. On the one hand, $U(m)$ is compact and, hence, can be turned into a Riemannian group; on the other hand,
since $SU(m)$ is the kernel of the Lie group homomorphism $\det:U(m)\rightarrow\mathbb S^1$, it is a closed Lie subgroup of $U(m)$, of
codimension $1$. As for the noncompact case, one can take $G$ and $M$ to be respectively equal to the complex orthogonal group $O(n,\mathbb C)$
and the complex special orthogonal group $SO(n,\mathbb C)$ (cf. Section $1.17$ of \cite{Knapp:02}).
} 
\end{remark}

\section{CMC hypersurfaces of Lorentzian groups}\label{section:CMC hypersurfaces of Lorentzian groups}

Along all of this section, $G^{n+1}$ is an $(n+1)-$dimensional Lorentzian group and $\varphi:M^n\to G^{n+1}$ is a connected spacelike
hypersurface of $G$, oriented by the choice of a globally defined timelike unit normal vector field $N$. In what follows, and unless 
otherwise stated, we stick to this context without further notice. We begin extending Theorem \ref{thm:main theorem} to this setting.

\begin{theorem}\label{thm:the compact Lorentzian case}
 Let $G^{n+1}$ be a Lorentzian group and $\varphi:M^n\rightarrow G^{n+1}$ be a compact connected oriented cmc spacelike hypersurface of $G$,
transversal to a timelike element $X$ of $\mathfrak g$. If the mean curvature $H$ of $\varphi$ satisfies
$$H^2\geq-n\text{\rm Ric}_G(N)$$
along $M$, then:
\begin{enumerate}
 \item[$(a)$] The Ricci curvature of $G$ vanishes in the direction of $N$.
 \item[$(b)$] $\varphi$ is totally geodesic
 \item[$(c)$] $\varphi$ is a lateral class of an embedded Lie subgroup of $G$.
\end{enumerate}
\end{theorem}

\begin{proof}
Without loss of generality, we can ask that $|X|=1$ and $f_X=\langle X,N\rangle<0$ on $M$. On the other hand, it follows from Lemma 
\ref{lemma:Laplacian of support functions} that
$$\Delta f_X=(|A|^2+\text{\rm Ric}_G(N))f_X.$$

Now, Cauchy-Schwarz inequality gives $|A|^2\geq\frac{H^2}{n}$, with equality only at umbilical points. Thus,
$$\Delta f_X\leq\left(\text{\rm Ric}_G(N)+\frac{H^2}{n}\right)f_X\leq 0,$$
so that $f_X$ is a superharmonic function on the compact Riemannian manifold $M$. Therefore, Hopf's theorem guarantees that $f_X$ is constant;
in turn, this gives $|A|^2=\frac{H^2}{n}$ and $\text{\rm Ric}_G(N)=\frac{H^2}{n}$ on $M$, from which it follows that $\varphi$ is totally 
umbilical.

Write $A=\lambda\text{Id}$, where $\lambda=-\frac{H}{n}$ and $\text{Id}$ refers to the identity homomorphism on $TM$. Also, fix an 
orthonormal basis $\mathcal B$ of $\mathfrak g$ in such a way that $X=X_{n+1}$. It follows from the above, (\ref{lemma:gradient of support functions}) and $c_{n+1,n+1}^i=0$ that
$$0=\nabla f_X=-\lambda X^{\top}+\frac{1}{2}\sum_{l=1}^n\left(\sum_{i=1}^{n+1}c_{l,n+1}^i\epsilon_lf_i\right)X_l^{\top},$$
so that
\begin{equation}\label{eq:auxiliar 1 to the compact Lorentzian case}
-\lambda X+\frac{1}{2}\sum_{l=1}^n\left(\sum_{i=1}^{n+1}c_{l,n+1}^i\epsilon_lf_i\right)X_l=\alpha N,
\end{equation}
for some function $\alpha:M\rightarrow\mathbb R$. Taking the scalar product of both members of the above equality with $X$, we get
$$\lambda=-\lambda\langle X,X\rangle=\alpha\langle N,X\rangle=\alpha f_X.$$

If $H\neq 0$, then $\alpha=\frac{\lambda}{f_X}=-\frac{H}{nf_X}$, a nonzero constant on $M$. On the other hand, by applying Lemma 
\ref{lemma:Laplacian of support functions} again, we get
$$\Delta f_i=(|A|^2+\text{\rm Ric}_G(N))f_i=0$$
for $1\leq i\leq n$, so that Hopf's theorem gives $f_i$ constant on $M$, for all such $i$. Then, 
(\ref{eq:auxiliar 1 to the compact Lorentzian case}) assures that $N$ is the restriction of an element of $\mathfrak g$ to $M$. 
Let $E$ be such element, and choose an orthonormal basis $\mathcal B'=(E_1,\ldots,E_n,E)$ of $\mathfrak g$. Then, the restrictions of 
$E_1,\ldots,E_n$ to $M$ are tangent to $M$ and, since $M$ is a hypersurface of $G$, we get $[E_i,E_j]_{|M}\in\mathfrak X(M)$, for all 
$1\leq i,j\leq n$. It thus follows that the distribution of $G$ generated by 
$E_1,\ldots,E_n$ is integrable. Therefore, an argument similar to that of item (c) of Theorem \ref{thm:main theorem} guarantees that
$M$ is a leaf of the corresponding foliation and, hence, totally geodesic. This contradicts the fact that $H\neq 0$. 

We conclude that $H=0$ and $\varphi$ is totally geodesic. By invoking again the argument of item (c) of Theorem \ref{thm:main theorem}, we 
finish the proof of item (c) and, thus, of the theorem. 
\end{proof}

In order to get a similar result in the noncompact case, we have to ask more from the Lorentzian group $G$, as well as to
control the size of the Gauss map of the hypersurface. 

We begin, by extending the notion of Gauss map to spacelike hypersurfaces of Lorentzian groups. To this end, 
let $X$ be a given timelike element of $\mathfrak g$. Since $T_eG$ is isometric to $\mathbb L^{n+1}$, we define the 
hyperbolic space of $T_eG$ with respect to $X_e$ as
$$\mathbb H^n(T_eG)=\{v\in T_eG;\,\langle v,v\rangle=-1\,\,\text{and}\,\,\langle v,X_e\rangle<0\}.$$

By changing $N$ to $-N$, if necessary, we may suppose that $N$ is in the same time-orientation as $X$, i.e., that
$f_X=\langle X,N\rangle<0$. Due to this choice, the Gauss map $\eta:M^n\rightarrow\mathbb H^n(T_eG)$ of $\varphi$ with respect to $N$ and 
$X$ is, then, defined formally in exactly the same way as was done right after Theorem \ref{thm:main theorem} for the Riemannian case. From
now on, unless explicitly stated otherwise, we assume that $f_X<0$.

The coming result extends, for Lorentzian groups, a theorem of Y. L. Xin (cf. \cite{Xin:91}) on complete spacelike hypersurfaces of 
$\mathbb L^{n+1}$. Prior to stating and proving it, we need the following auxiliar result.

\begin{lemma}
Let $G^{n+1}$, $n\geq 2$, be a Lorentzian group and $\varphi:M^n\rightarrow G^{n+1}$ be a connected cmc spacelike hypersurface of $G$, 
transversal to a timelike element $X$ of $\mathfrak g$ and oriented by the choice of a timelike unit vector field $N$. If the image of 
the Gauss map of $\varphi$ with respect to $N$ and $X$ is bounded, then
$$\inf_M\text{\rm Ric}_G(N)>-\infty.$$
\end{lemma}

\begin{proof}
Let $(X_1,\ldots,X_n,X_{n+1}=X)$ be an orthonormal basis of $\mathfrak g$ and write $N=-f_XX+\sum_{i=1}^nf_iX_i$,
so that 
$$-f_X^2+\sum_{i=1}^nf_i^2=-1$$
and
\begin{equation}\nonumber
\begin{split}
\text{Ric}_G(N)
&\,=\text{Ric}_G\left(-f_XX+\sum_{i=1}^nf_iX_i,-f_XX+\sum_{j=1}^nf_jX_j\right)\\
&\,=f_X^2\text{Ric}_G(X)+\sum_{i,j=1}^nf_if_j\text{Ric}_G(X_i,X_j)\\
&\,\,\,\,\,\,\,-2f_X\sum_{i=1}^nf_i\text{Ric}_G(X,X_i).
\end{split}
\end{equation}

Since $f_X<0$, we get from the triangle inequality and Cauchy-Schwarz inequality that
\begin{equation}\nonumber
\begin{split}
\text{Ric}_G(N)&\,\geq f_X^2\text{Ric}_G(X)-\sum_{i,j=1}^n|f_if_j||\text{Ric}_G(X_i,X_j)|\\
&\,\,\,\,\,\,\,+2f_X\sum_{i=1}^n|f_i||\text{Ric}_G(X,X_i)|\\
&\,\geq f_X^2\text{Ric}_G(X)-\max_{1\leq i,j\leq n}|\text{Ric}_G(X_i,X_j)|\sum_{i,j=1}^n|f_if_j|\\
&\,\,\,\,\,\,\,+2f_X\sum_{i=1}^n|f_i||\text{Ric}_G(X,X_i)|\\
&\,\geq f_X^2\text{Ric}_G(X)-\max_{1\leq i,j\leq n}|\text{Ric}_G(X_i,X_j)|\left(\sum_{i=1}^n|f_i|\right)^2\\
&\,\,\,\,\,\,\,+2f_X\left(\sum_{i=1}^n|f_i|\right)^{1/2}\left(\sum_{i=1}^n|\text{Ric}_G(X,X_i)|\right)^{1/2}\\
&\,\geq f_X^2\text{Ric}_G(X)-n\cdot\max_{1\leq i,j\leq n}|\text{Ric}_G(X_i,X_j)|\sum_{i=1}^nf_i^2\\
&\,\,\,\,\,\,\,+2nf_X\left(\sum_{i=1}^nf_i^2\right)\left(\sum_{i=1}^n|\text{Ric}_G(X,X_i)|\right)^{1/2}\\
\end{split}
\end{equation}

Now, it is enough to observe that all of $\text{Ric}_G(X)$, $\text{Ric}_G(X_i,X_j)$ and $\text{Ric}_G(X,X_i)$ are constant on $G$, substitute 
$\sum_{i=1}^nf_i^2=f_X^2-1$ and take into account that the boundedness of the image of the Gauss map is equivalent to the fact that 
\begin{equation}\label{eq:definition of c}
c:=-\inf_M\langle X,N\rangle<+\infty
\end{equation}
and, hence, that $-c\leq f_X\leq-1$.
\end{proof}

\begin{theorem}\label{thm:extending a theorem of Xin}
Let $G^{n+1}$, $n\geq 2$, be a Lorentzian group with sectional curvatures bounded from above on Lorentzian planes.
Let $\varphi:M^n\rightarrow G^{n+1}$ be a complete connected oriented cmc spacelike hypersurface of $G$,
transversal to a timelike element $X$ of $\mathfrak g$. If the image of the Gauss map of $\varphi$ with respect to $N$ and $X$ is bounded and
the mean curvature $H$ of $\varphi$ is such that
$$H^2\geq-n\inf_M\text{\rm Ric}_G(N),$$
then $\varphi$ is totally umbilical and $H^2=-n\inf_M\text{\rm Ric}_G(N)$. Moreover, if $\varphi$ is not totally geodesic, then:
\begin{enumerate}
\item[$(a)$] $\varphi(M)$ is a lateral class of a Lie subgroup of $G$;
\item[$(b)$] $M$ is isometric to a hyperbolic space form.
\end{enumerate}
\end{theorem}

\begin{proof}
Let $\lambda=\inf_M\text{\rm Ric}_G(N)$, so that $\lambda>-\infty$, by the previous lemma.
As in the proof of Theorem \ref{thm:the compact Lorentzian case}, %and with notations as in (\ref{eq:recalling what an Einstein metric is})
we get
\begin{equation}\label{eq:auxiliar to extending a theorem of Xin}
\Delta f_X\leq\left(\frac{H^2}{n}+\lambda\right)f_X,
\end{equation}
with equality at some point of $M$ if, and only if, $\varphi$ is totally umbilical.

In order to analyse the differential inequality above, use the boundedness of the Gauss map of $\varphi$ to choose $c$ as in
(\ref{eq:definition of c}). Suppose, for the moment, that $M$ has Ricci curvature bounded from below (we shall soon justify this claim); 
since it is complete, we can invoke Omori-Yau's maximum principle (cf. \cite{Yau:75}) to get a sequence $(p_k)_{k\geq 1}$ of points of $M$, such that 
$$f_X(p_k)\stackrel{k}{\longrightarrow}-c\ \ \text{and}\ \ \liminf_k\Delta f_X(p_k)\geq 0.$$
If we now remember that $\frac{H^2}{n}+\lambda\geq 0$, evaluate (\ref{eq:auxiliar to extending a theorem of Xin}) at $p_k$ and make 
$k\to+\infty$, we get
$$0\leq\liminf_k\Delta f_X(p_k)\leq\left(\frac{H^2}{n}+\lambda\right)f_X(p_k)\leq-c\left(\frac{H^2}{n}+\lambda\right).$$
This is a contradiction, unless $\frac{H^2}{n}+\lambda=0$, so that $\varphi$ is totally umbilical.

If $H=0$, then $\varphi$ is totally geodesic and, from (\ref{eq:differential of the Gauss map}), the Gauss map $\eta$ is constant on $M$. Thus,
(\ref{eq:expression for the Gauss map}) guarantees that the unit normal vector field $N$ is the restriction of an element of $\mathfrak g$
to $M$. From this point on, if we reason as in the next to last paragraph in the proof of Theorem \ref{thm:the compact Lorentzian case}, 
we conclude that $\varphi(M)$ is a lateral class of a Lie subgroup of $G$.

If $H\neq 0$, it follows from (\ref{eq:differential of the Gauss map}) that $(\eta_p)_*=-\frac{H}{n}((L_{p^{-1}})_*)_p,$
for all $p\in M$ and, hence, $\eta:M^n\to\mathbb H^n(T_eG)$ is a local homothety. Therefore, $M^n$ has constant negative sectional curvatures 
and, thus, is a hyperbolic space form.

To what was left to prove, let $p\in M$ be given and $(e_1,\ldots,e_n)$ be an orthonormal basis of $T_pM$. For a unit vector $v\in T_pM$, it 
follows from Gauss' equation, Cauchy-Schwarz inequality and the previous lemma that
\begin{equation}\nonumber
 \begin{split}
  \text{Ric}_M(v)&\,=\text{Ric}_G(N_p)-K_G(v,N_p)+\frac{H}{n}\langle Av,v\rangle+\sum_{k=1}^n\langle Av,e_k\rangle^2\\
&\,\geq\inf_M\text{Ric}_G(N)-K_G(v,N_p)-\frac{|H|}{n}|Av|+|Av|^2\\
&\,\geq\inf_M\text{Ric}_G(N)-\sup_{\sigma\,\,\text{Lorentz}}K_G(\sigma)-\frac{H^2}{4n}>-\infty.
 \end{split}
\end{equation}

\end{proof}

\begin{remarks}
{\em 
We shall now give some classes of examples to which we can apply Theorems \ref{thm:the compact Lorentzian case} and 
\ref{thm:extending a theorem of Xin}.\\

\noindent (i) The prototype example related to theorems \ref{thm:the compact Lorentzian case} and \ref{thm:extending a theorem of Xin} is a product 
Lorentzian group
$$G^{n+1}=-\mathbb R\times H^n,$$ 
i.e., the direct product of $\mathbb R$ and an $n-$dimensional Riemannian group $H^n$, furnished with the usual Lorentzian product structure.
Indeed, let $\partial_t$ denote the canonical basis of $\mathbb R$ and $\{X,Y\}$ be a set of orthonormal elements of the Lie algebra 
$\mathfrak h$ of $H$; the timelike directions on $G$ are those generated by $\partial_t$ or $a\partial_t+X$, where $a\in\mathbb R$ is such that
$a^2>1$. Therefore, if $\sigma$ is a Lorentzian plane on $G$, then $\sigma$ admits an orthogonal basis of one of the following two forms: 
$(\partial_t,Y)$ or $(a\partial_t+X,Y)$, where $a$, $X$ and $Y$ are as above. This being said, it follows from 
(\ref{eq:sectional curvature}) that
$$K_G(\partial_t,Y)=0\ \ \text{and}\ \ K_G(a\partial_t+X,Y)=\frac{K_H(X,Y)}{4(1-a^2)}\leq 0.$$

As it happens to all lateral classes of subgroups of semi-Riemannian groups, the slices $M_t^n:=\{t\}\times H^n$ are totally geodesic 
hypersurfaces of $G$. In our case, $M_t$ is also spacelike, oriented by the unit normal vector field $\partial_t$, such that 
$\text{Ric}_G(\partial_t)=0$. Finally, $M_t$ is compact (resp. complete) whenever $H^n$ is compact (resp. complete).\\

%we have
%$$\text{Ric}_G(a\partial_t+X)=\sum_{k=1}^n\frac{K_H(X,X_j)}{4(1-a^2)},$$
%where $(X_1,\ldots,X_n)$ is an orthonormal basis for $\mathfrak h$ such that $X_1=X$; this, in turn, goes to $-\infty$ if $\mathfrak h$ 
%is not abelian and $|a|\to 1^+$. Therefore, in what concerns Theorem \ref{thm:extending a theorem of Xin}, this example works only if $H$, and
%then $G$, is abelian.\\

\noindent (ii) With respect to nonabelian Lorentzian groups, it was proved by A. Medina (cf. \cite{Medina:85}) that the only nonabelian simply 
connected solvable Lie groups which can be turned into Lorentzian groups are the oscillator groups $G_m$ ($m\in\mathbb N$) 
and their direct products with Euclidean spaces. 

For a given $m\in\mathbb N$, the $(2m+2)-$dimensional Lie group $G_m$ can be succintly described as follows (cf. \cite{Gadea:99}): its 
Lie algebra $\mathfrak g_m$ has basis $(P,X_1,\ldots,X_m,Y_1,\ldots,Y_m,Q)$, with nontrivial brackets $[X_i,Y_j]=P$, $[Q,X_j]=Y_j$ and 
$[Q,Y_j]=-X_j$. The nontrivial scalar products of the corresponding bi-invariant metric tensor are
$\langle P,Q\rangle=1$ and $\langle X_i,X_j\rangle=\langle Y_i,Y_j\rangle=\delta_{ij}$, so that
 an orthonormal basis for $\mathfrak g_m$ is given by $(U,X_1,\ldots,X_m,Y_1,\ldots,Y_m,V)$, where $U=\frac{1}{\sqrt 2}(P+Q)$ and
$V=\frac{1}{\sqrt 2}(P-Q)$. The timelike directions of $\mathfrak g_m$ are those generated by $V$, $aV+U$, $aV+X_i$ or $aV+Y_i$, for some
$1\leq i\leq m$ and $a\in\mathbb R$ such that $|a|>1$. By using (\ref{eq:sectional curvature}) again, we can promptly check that
$$K_{G_m}(V,U)=0,\ \ K_{G_m}(V,X_i)=K_{G_m}(V,Y_i)=-\frac{1}{8},$$
$$K_{G_m}(aV+U,X_i)=K_{G_m}(aV+U,X_i)=\frac{1-a}{8(1+a)},$$
$$K_{G_m}(aV+X_i,U)=\frac{1}{8(1-a^2)}$$
and
$$K_{G_m}(aV+X_i,Y_j)=K_{G_m}(aV+Y_i,X_j)=\frac{a^2}{8(1-a^2)}.$$
Then, as it happened in (a), the sectional curvature of $G_m$ along Lorentzian planes are bounded from above.

Finally, it is immediate to verify that the vector subspace $\mathfrak h$ of $\mathfrak g_m$, generated by $U$, $X_1$, \ldots, $X_m$, $Y_1$, 
\ldots, $Y_m$, is a Lie subalgebra of $\mathfrak g_m$, hence defines a Lie subgroup $H^{2m+1}$ of $G_m$. As observed by the end of item (a), 
$H^{2m+1}$ (as well as all of its lateral classes) is a totally geodesic hypersurface of $G_m$; moreover, it is spacelike (for it is 
oriented by the unit normal timelike vector field $V$), complete (as occurs to all Riemannian homogeneous spaces), and such that the 
Ricci curvature of $G_m$ in the direction of $V$ is $-\frac{m}{4}$.
}
\end{remarks}


\begin{thebibliography}{20}

\bibitem{Albujer:11}
A. L. Albujer, F. E. C. Camargo and H. F. de Lima. 
{\em Complete spacelike hypersurfaces in a Robertson-Walker spacetime}. 
Math. Proc. of the Cambridge Phil. Soc. 151 (2011), 271-282.

\bibitem{Caminha:11}
F. E. C. Camargo, A. Caminha and H. F. de Lima. 
{\em Bernstein-type theorems in semi-Riemannian warped products}. 
Proc. Amer. Math. Soc. 139 (2011), 1841-1850. 

\bibitem{Caminha:10}
F. E. C. Camargo, A. Caminha and P. Sousa. 
{\em Complete foliations of space forms by hypersurfaces}. 
Bull. Braz. Math. Soc. 41 (2010), 339-353.

\bibitem{Chevalley:46}
C. Chevalley.
{\em Theory of Lie Groups I}.
Princeton Univ. Press, Princeton, 1946.

\bibitem{doCarmo:92}
M. P. do Carmo.
{\em Riemannian Geometry}.
Birkh\"auser, Boston, 1992.

%\bibitem{Daniel:07}
%B. Daniel.
%{\em Isometric immersions into 3-dimensional homogeneous manifolds}.
%Comment. Math. Helv. 82 (2007), 87–131.

\bibitem{Schoen:80} 
D. Fischer-Colbrie and R. Schoen.
{\em The Structure of Complete Stable Minimal Surfaces in $3-$Manifolds of Nonnegative Scalar Curvature}.
Comm. Pure and Appl. Math. 33 (1980), 199-211.

\bibitem{Gadea:99}
P. M. Gadea and J. A. Oubi\~na.
{\em Homogeneous Lorentzian structures on the oscillator groups}. 
Arch. Math. 73 (1999), 311–320.

%\bibitem{Mira:12}
%B. Daniel and P. Mira.
%{\em Existence and uniqueness of constant mean curvature spheres in $\text{Sol}_3$}.
%J. f\"r die reine und angewandte Mathematik. Accepted for publication (2012).

\bibitem{Humphreys:97}
J. E. Humphreys. 
{\em Introduction to Lie Algebras and Representation Theory}.
Springer Verlag, Nova Iorque, 1997.

\bibitem{Knapp:02}
A. W. Knapp.
{\em Lie Groups Beyond an Introduction}.
Birkh\"auser, Rensselaer, 2002.

\bibitem{Lee:02}
J. M. Lee.
{\em Introduction to Smooth Manifolds}.
Springer Verlag, Nova Iorque, 2002.

%\bibitem{Lira:12}
%J. H. S. de Lira and M. F. de Melo.
%{\em Existence of isometric immersions into nilpotent Lie groups}.
%Geom. Dedicata 157 (2012), 339-365.

\bibitem{Medina:85}
A. Medina.
{\em Groupes de Lie munis de m\'etriques bi-invariantes}.
T\^ohoku Math. J. 37 (1985), 405-421.

\bibitem{O'Neill:83}
B. O'Neill.
{\em Semi-Riemannian Manifolds with Applications to Relativity}.
Academic Press, New York, 1983.

\bibitem{Ripoll:03}
N. do Esp\'{\i}rito Santo, S. Fornari, K. Frensel and J. Ripoll.
{\em Constant mean curvature hypersurfaces in a Lie group with a bi-invariant metric}.
Manuscripta Math. 111 (2003), 459-470.

\bibitem{Xin:03}
Y. L. Xin.
{\em Minimal Submanifolds and Related Topics}.
World Scientific Pul. Co., London, 2003.

\bibitem{Xin:91}
Y. L. Xin.
{\em On the Gauss image of a spacelike hypersurface with constant mean curvature in Minkowski space}.
Comm. Math. Helv. 66 (1991), 590-598.

\bibitem{Yau:75} 
S. T. Yau.
{\em Harmonic functions on complete Riemannian Manifolds}.
Comm. Pure and Appl. Math. 28 (1975), 201-228.

\bibitem{Yau:76} 
S. T. Yau.
{\em Some Function-Theoretic Properties of Complete Riemannian Manifolds and their Applications to Geometry}.
Indiana Univ. Math. J. 25 (1976), 659-670.

\end{thebibliography}
\end{document}